\documentclass[12pt]{amsart}

\usepackage{hyperref}
\usepackage{ifthen}
\usepackage{amssymb,amsmath,stmaryrd}
\usepackage[all,cmtip]{xy}

\newtheorem{theorem}{Theorem}[section]
\newtheorem{lemma}[theorem]{Lemma}
\newtheorem{proposition}[theorem]{Proposition}

\theoremstyle{definition}
\newtheorem{definition}[theorem]{Definition}

\theoremstyle{remark}
\newtheorem{remark}[theorem]{Remark}

\numberwithin{equation}{section}

\newcommand{\Map}{\mathrm{Map}}

\newcommand{\Id}{\mathrm{Id}}
\newcommand{\id}{\mathrm{id}}
\newcommand{\Comm}{\mathrm{Comm}}
\newcommand{\Ind}{\mathrm{Ind}}

\newcommand{\Mod}{\mathrm{-Mod}}
\newcommand{\fib}{\mathrm{hofib}}
\newcommand{\Topp}{\mathrm{Top}_\ast}

\newcommand{\cO}{\mathcal{O}}

\newcommand{\cS}{\mathcal{S}}
\newcommand{\cR}{\mathcal{R}}
\newcommand{\cL}{\mathcal{L}}
\newcommand{\cA}{\mathcal{A}}

\newcommand{\Qs}{\bar{\cR}}
\newcommand{\Q}{\bar{Q}}

\newcommand{\cSp}{\cS p}

\newcommand{\idop}[1][\ast]{\partial_{#1}(\Id)}

\newcommand{\sm}{\wedge}

\newcommand{\D}{\mathbb{D}}
\newcommand{\F}{\mathbb{F}}

\newcommand{\HF}[1][0]{\ifthenelse{\equal{#1}{0}}{}{\Sigma^{#1}}\mathrm{H}\mathbb{F}_2}

\newcommand{\sym}[2][2]{{#2}^{{\sm}#1}_{h\Sigma_{#1}}}
\newcommand{\symt}[2][2]{{#2}^{{\tp}\,#1}_{h\Sigma_{#1}}}

\newcommand{\opsusp}{\mathbf{s}}
\newcommand{\Lie}{\mathsf{Lie}}
\newcommand{\sLie}{\opsusp\Lie}

\newcommand{\frQsLie}{\opsusp\cL_{\Qs}}
\newcommand{\aQs}{\cA_{\Qs}}
\newcommand{\ti}{\mathrm{ti}}

\newcommand{\tp}{\otimes}

\newcommand{\fb}[1]{\llbracket #1 \rrbracket}

\title{The mod $2$ homology of free spectral Lie algebras}
\author{Omar Antol\'{\i}n Camarena}
\address{Institute of Mathematics, UNAM, Mexico City}
\email{omar@matem.unam.mx}


\begin{document}

\begin{abstract}
  The Goodwillie derivatives of the identity functor on pointed spaces
  form an operad $\idop$ in spectra. Adapting a definition of Behrens,
  we introduce mod $2$ homology operations for algebras over this
  operad and prove these operations account for all the mod $2$
  homology of free algebras on suspension spectra of simply-connected
  spaces.
\end{abstract}

\maketitle
\tableofcontents

\section{Introduction}

Goodwillie calculus \cite{GoodwillieIII} associates to appropriate
functors $F : \Topp \to \Topp$ a tower of approximations
$$ \cdots \to P_n F \to P_{n-1} F \to \cdots \to P_1 F \to P_0 F $$
that is analogous to the sequence of Taylor polynomials for functions
of a real variable. The homotopy fibers
$D_n F = \fib (P_n F \to P_{n-1} F)$ are called the layers of the
Goodwillie tower and are analogous to individual monomials
${f^{(n)}(0)x^n}/{n!}$ in the Taylor expansion of a function.
Goodwillie proved that these layers are of the form
$D_n F (X) = \Omega^\infty(\partial_n F \sm \Sigma^\infty X^{\sm
  n})_{h\Sigma_n}$ for some sequence of spectra $\partial_n F$ where
the $n$-th spectrum is equipped with an action of $\Sigma_n$.

These derivatives $\partial_n F$ are very interesting even for
$F = \Id$ and have been much studied in that case; they can be
described as the Spanier--Whitehead duals of certain finite complexes.
The first such description was obtained by Johnson \cite{Johnson}; a
second description is in terms of partitions and appears in
\cite{AroneMahowald}. The \emph{partition complex} $P_n$ is the
pointed simplicial set
\[N\Pi_n \!\Bigm/\! \bigl(N(\Pi_n \setminus \{\hat{0}\}) \cup
  N(\Pi_n \setminus \{\hat{1}\})\bigr), \]
where $\Pi_n$ is the poset of partitions of a set with $n$ elements,
ordered by refinement; $\hat{0}$ and $\hat{1}$ denote its least and
greatest element, respectively; and $N$ denotes, as usual, the nerve
functor. We shall regard $P_n$ as having the action of $\Sigma_n$
induced by permutations of the $n$-element set.

The layers of the Goodwillie tower of the identity are given by
\[D_n(\Id)(X) = \Omega^\infty \left( \Map_\ast(P_n, \Sigma^\infty
  X^{\sm n})_{h\Sigma_n} \right),\]
where $\Map_\ast$ denotes the spectrum of maps from a pointed space to
a spectrum. This implies that $\idop[n]$ is $\Map_\ast(P_n, S)$, the
Spanier--Whitehead dual of $P_n$.

In \cite{ChingBar}, Ching constructs an operad structure on $\idop$
that is easiest to describe in dual form: as a cooperad structure on
$P_\ast$. That cooperad is the bar construction on the nonunital
commutative operad in spectra (given by $\Comm_n = S$ for all
$n \ge 1$), so that the operadic suspension of $\idop$ is Koszul dual
to the commutative operad and we can think of $\idop$ as a shifted
version of the Lie operad. Alternatively, one can also see a relation
between $\idop$ and the Lie operad using what is known about the
homology of the partition complex, namely that the space of $n$-ary
operations of the Lie operad in Abelian groups is isomorphic as a
$\mathbb{Z}[\Sigma_n]$-module to
$\mathrm{Hom}(H_{n-2}(P_n), \mathrm{sgn})$ ---where $\mathrm{sgn}$ is
the sign representation of $\Sigma_n$.

The mod $p$ homology of the layers,
\[\D_n(X) := \D_n(\Id)(X) = \left( \idop[n] \sm \Sigma^\infty X^{\sm n}
  \right)_{h\Sigma_n},\]
was studied in \cite{AroneMahowald} in the case that $X$ is a sphere.
(Since the free $\idop$-algebra on a space $X$ is given by
$\bigoplus_{n \ge 0} \D_n(X)$, the results in that paper can be
interpreted as being about the mod $p$ homology of the free
$\idop$-algebra on $S^n$.) In the case $p=2$, for example, what Arone
and Mahowald showed is that $H_\ast(\D_n(S^m); \F_p)$ is only non-zero
when $n=2^k$ is a power of $2$ and in that case it is
$\Sigma^{-k} CU_\ast$ as a module over the Steenrod algebra, where
$CU_\ast$ is the free graded $\F_p$-vector space with basis given by
the ``completely unadmissible'' words of length $k$:
$$ \{Q^{s_1} \cdots Q^{s_k} u : s_k \ge m, s_i > 2s_{i+1}\} $$ where
$u$ is a generator of $H_m(S^m; \F_p)$ and the action of the Steenrod
algebra on $CU_\ast$ is given by the Nishida relations.

Behrens \cite{Behrens} uses this computation to introduce mod $2$
homology operations
$\bar{Q}^j : H_d(D_iF(X)) \to H_{d+j-1}(D_{2i}F(X))$ for $j \ge d$ on
the layers of the Goodwillie tower of a functor $F$. We adapt his
definition to produce homology operations on $\idop$-algebras (see
definition \ref{d:unops}). Behrens shows the Arone--Mahowald
computation can be interpreted as saying that the homology of
$\bigoplus_{n \ge 0} \D_n(X)$ has an $\F_2$-basis consisting of
completely unadmissible sequences of $\bar{Q}^j$'s with excess at
least $k$ applied to the fundamental class of $S^k$; furthermore, he
computes the relations satisfied by the $\bar{Q}^j$'s.

In the present work we compute the mod $2$ homology of the free
$\idop$-algebra on a spectrum, showing that it is roughly speaking the
free module over the algebra of operations $\Q^j$'s on the free Lie
algebra on $H_\ast(X)$ (see theorem \ref{t:homfree}).

\subsection{Some related work}

Since the first appearance of this work as the author's PhD thesis, an
analogous computation of the mod $p$ homology for odd primes $p$ has
been carried out by Kjaer \cite{Kjaer}. In that setting the story is
less complete: due to a lack of an analogue of Priddy's trace formula
\cite[Lemma 1.4.3]{Behrens}, the relations between the homology
operations have yet to be fully determined.

In the upcoming paper \cite{Lukas}, Brantner computes the
$E(n)$-homology of certain free $\idop$-algebras in the category of
$K(n)$-local spectra and thus computes the $E(n)$-homology operations
on $K(n)$-local $\idop$-algebras. Here $E(n)$ and $K(n)$ denote Morava
$E$-theory and Morava $K$-theory respectively.

Also, in the final section of the paper we will explain the relation
between our computation and the $E_1$-page of Kuhn's spectral sequence
for Topological Andr\'e--Quillen homology of an augmented
$E_\infty$-ring spectrum \cite[Theorem 8.1]{Kuhn}.

\section*{Acknowledgments}

I would like to thank Jacob Lurie for suggesting this general topic
and for answering many questions, Lukas Brantner for many useful
conversations and particularly for correcting an error in an earlier
version (see remark \ref{r:lukas}), and Nick Kuhn for telling me about
his paper \cite{Kuhn} (see section \ref{s:divpow}). I would also like
to thank two anonymous referees who both made several useful remarks
which improved the clarity of this paper.

\section{Homology operations on algebras for operads}

We will work in a symmetric monoidal category of spectra, such as EKMM
$S$-modules \cite{EKMM}, taking ``spectrum'' to mean $S$-module and
``$E_\infty$-ring spectrum'' to mean commutative $S$-algebra. This is
the same framework used in \cite{ChingBar} to put an operad structure
on the derivatives of the identity.

Given an operad $\cO$ in spectra we will denote by $F_{\cO}$ the free
$\cO$-algebra functor. This functor is a monad, and $\cO$-algebras are
equivalently algebras for it. If $E$ is an $E_\infty$-ring spectrum,
then there is an operad in $E$-module spectra we will denote by
$E \sm \cO$, and a free $(E \sm \cO)$-algebra functor $F_{E \sm \cO}$
defined on $E$-module spectra. The $E$-module of $n$-ary operations in
$E \sm \cO$ is the free $E$-module on the spectrum $\cO_n$:
$(E \sm \cO)_n = E \sm \cO_n$; we get an operad structure on
$E \sm \cO$ induced from the operad structure on $\cO$ because the free
$E$-module functor is symmetric monoidal. The free algebra functors
are related in the expected way:
$E \sm F_{\cO}(X) \simeq F_{E \sm \cO}(E \sm X)$. We will also make use
of the functor between $\cO$-algebras and $(E \sm \cO)$-algebras
induced by the free $E$-module functor, $E \sm -$.

We will only consider cofibrant operads for which the notion of
algebra is homotopy invariant, meaning that we can think of the
homotopy type of the free $\cO$-algebra as being given by
\[F_\cO(X) = \bigvee_{n \ge 0} (\cO_n \sm X^n)_{h\Sigma_n},\]
and that we will think of an $\cO$-algebra structure on $A$ as
providing maps $(\cO_n \sm A^{\sm n})_{h\Sigma_n} \to A$.

Every class
$\alpha \in E_m\left(F_\cO\left(\bigvee_{i=1}^k S^{d_i}\right)\right)$
in the $E$-homology of the free $\cO$-algebra on a wedge of $k$
spheres gives a $k$-ary homology operation on the $E$-homology of any
$\cO$-algebra $A$, defined as follows.

Given $x_i \in E_{d_i}(A)$ ($i=1, \ldots, k$), we can represent each
$x_i$ by a map of spectra $S^{d_i} \to E \sm A$, and thus the whole
collection of them can be described by a single map of spectra
$\bar{x} : \bigvee_{i=1}^k S^{d_i} \to E \sm A$. Since $E \sm A$ is
an $(E \sm \cO)$-algebra, $\bar{x}$ has an adjoint $\tilde{x}$ which
is a map of $(E \sm \cO)$-algebras to $E \sm A$ from the free
$(E \sm \cO)$-algebra on $\bigvee_{i=1}^k S^{d_i}$, namely
$F_{E \sm \cO}(\bigvee_{i=1}^k \Sigma^{d_i} E) = E \sm
F_\cO(\bigvee_{i=1}^k S^{d_i})$.

The homology operation corresponding to $\alpha$, is
$\alpha_\ast : \bigotimes_{i=1}^k E_{d_i}(A) \to E_m(A)$ defined by
setting $\alpha_\ast(x_1 \otimes \cdots \otimes x_k)$ to be
represented by the map
\[S^m \xrightarrow{\alpha} E \sm F_\cO(\bigvee_{i=1}^k S^{d_i})
\xrightarrow{\tilde{x}} E \sm A.\]

An analogous construction gives operations on the stable homotopy of
$(E \sm \cO)$-algebras. Given an $\cO$-algebra $A$, the operations on
the $E$-homology of $A$ coincide with those produced on the homotopy
of the $(E \sm \cO)$-algebra $E \sm A$.

To get a useful theory of homology operations for $\cO$-algebras,
besides computing those homology groups, the various
$E_m(F_\cO(\bigvee_{i=1}^k S^{d_i}))$, one must organize the
operations: find a relatively small collection of operations that
generate all others and find a generating set of relations for the
operations. This has been carried out for $\mathrm{H}\F_p$-homology of
algebras for the $E_n$-operads, due to May in the case $n=\infty$, and
due to F. Cohen in the case $1 \le n < \infty$; see
\cite{CohenLadaMay}.

Homology operations with field coefficients are simpler to study,
because of the following result:

\begin{proposition}\label{p:field}
  Let $\cO$ be an operad in spectra. The homology with coefficients in
  a field $k$ of the free $\cO$-algebra on a spectrum $X$ is a functor
  of the homology of $X$.
\end{proposition}

\begin{proof}
  Consider the following commutative diagram:

  \[ \xymatrix{ \cSp \ar[d]_{F_\cO} \ar[rr]^-{Hk\sm-} && Hk\Mod
  \ar[d]_{F_{Hk\sm\cO}} \ar[r]^-\pi & D(k) \ar[r]^-\cong
  \ar[d]^{\hat{F}}
  & \mathrm{GrVect}_k \ar[d]^{\tilde{F}} \\
  \cSp \ar[rr]_-{Hk\sm-} && Hk\Mod \ar[r]_-\pi & D(k) \ar[r]_-\cong &
  \mathrm{GrVect}_k} \]

Here $\cSp$ denotes the category of spectra, $Hk\Mod$ denotes the
category of $Hk$-module spectra, $D(k)$ is the homotopy category of
$Hk\Mod$ or, equivalently, the unbounded derived category of vector
spaces over $k$ and $\mathrm{GrVect}_k$ is the category of graded
vector spaces over $k$.

The functor $\pi$ is the projection from $Hk\Mod$ to its homotopy
category; this functor preserves coproducts but when the
characteristic of $k$ is not $0$, it does not send homotopy quotients
by the action of $\Sigma_n$ to quotients by the action of $\Sigma_n$,
so the induced monad $\hat{F}$ is no longer the free algebra functor
for an operad. Finally, when $k$ is a field there is an equivalence
$D(k) \cong \mathrm{GrVect}_k$, allowing us to define the monad
$\tilde{F}$ so that the last square commutes.
\end{proof}

\section{The spectral Lie operad and its desuspension}

Recall that the suspension of an operad $\cO$ in spectra is an operad
$\opsusp \cO$ defined so that:

\begin{itemize}
\item $\opsusp \cO$-algebra structures on $\Sigma A$ correspond to
  $\cO$-algebra structures on $A$,
\item the free algebra functors satisfy
  $F_{\opsusp \cO}(\Sigma X) = \Sigma F_{\cO}(X)$, and
\item as a symmetric sequence, $(\opsusp \cO)_n$ is given by
  $(S^{-1})^{\sm n} \sm \Sigma \cO_n$ with $\Sigma_n$ acting
  diagonally, permuting the smash factors on the left and acting on
  $\Sigma \cO_n$ via the suspension of the action on $\cO_n$ (that is,
  it acts trivially on the suspension coordinate of $\Sigma \cO_n$).
\end{itemize}

Following existing nomenclature, we will call the operad $\idop$
formed by the Goodwillie derivatives of the identity, the
\emph{spectral Lie operad}. Even though it is actually the
desuspension $\opsusp^{-1} \idop$ that is most closely analogous to
the classical Lie operad and some of our formulas would be simpler for
it, we will stick to the language of the $\idop$-operad and
$\idop$-algebras to make using the available literature easier. As a
symmetric sequence, $\opsusp^{-1} \idop$ is given by the derivatives
of the functor $\Omega \Sigma : \Topp \to \Topp$ (see \cite[Section
8]{GoodwillieIII}).

As we said before, the easiest way to describe the operad structure of
$\idop$ is to describe a cooperad structure on the bar construction of
the nonunital commutative operad, and obtain the operad structure of
$\idop$ by taking Spanier--Whitehead duals. For a description of
Ching's cooperad structure, we refer the reader to \cite[Section
4]{ChingBar}.

\section{Two kinds of Lie algebras in characteristic 2}
\label{s:liealg}

In this section we collect a few definitions about (graded) Lie
algebras we will need later. We will actually need to use two
different notions of Lie algebras. The usual definition of Lie algebra
over a field of characteristic $0$ is equivalent to being an algebra
for an operad $\Lie$ in Abelian groups. One can take algebras for that
operad in any category that is tensored over Abelian groups, such as
the category of $R$-modules or of graded $R$-modules for a commutative
ring $R$, and this gives one possible definition of graded Lie
algebra. Since $\idop$ is the suspension of the spectral version of
the Lie operad, we are also interested in algebras for the suspension
$\sLie$.

Spelling out the structure we see that a graded $\Lie$-algebra $L$
over a commutative ring $R$ is a graded module equipped with a binary
operation $[-,-] : L_i \otimes L_j \to L_{i+j}$ satisfying, for
homogeneous elements $x$, $y$ and $z$ of degrees $|x|$, $|y|$ and $|z|$:
\begin{itemize}
\item anti-symmetry, $[x,y] = -(-1)^{|x||y|}[y,x]$, and
\item the Jacobi identity,
  \[(-1)^{|z||x|}[x,[y,z]] + (-1)^{|y||x|}[y,[z,x]] +
  (-1)^{|z||y|}[z,[x,y]]=0.\]
\end{itemize}

For $\sLie$-algebras things are only slightly different:
\begin{itemize}
\item The bracket has degree $-1$: $[-,-] : L_i \otimes L_j \to L_{i+j-1}$.
\item Anti-symmetry becomes graded commutativity:
  \[[x,y] = (-1)^{|x||y|}[y,x].\]
\item The Jacobi identity stays the same!
\end{itemize}

All the signs in the above formulas come from the Koszul sign rule,
that is, from the signs in the symmetry isomorphism of the category of
graded $R$-modules. Since we will work over $R = \F_2$ we need not
worry about signs, but we mention them to point out that for an
element $x$ of even degree in a $\Lie$-algebra (or of odd degree in a
$\sLie$-algebra), the definitions imply that $2[x,x]=0$, but they do
not actually imply $[x,x]=0$ if $2$ is not invertible in $R$.

If $R$ has characteristic $2$, while $[x,x]$ may not be $0$, we do
have that any brackets involving it are $0$: by the Jacobi identity,
\[ [[x,x],y] = [[x,y],x] + [[y,x],x] = 2[[x,y],x] = 0.\]

As an example showing $[x,x]$ can be nonzero, the free $\Lie$-algebra
over $\F_2$ on one generator $x$ in an even degree is easily seen to
have basis $\{x, [x,x]\}$.

Given a graded associative $R$-algebra $A$, the graded commutator
$[x,y] = xy - (-1)^{|x||y|}yx$ gives $A$ the structure of a
$\Lie$-algebra, but all the algebras produced this way necessarily
have $[x,x] = 0$ for $|x|$ even. This means that if a $\Lie$-algebra
over an $R$ of characteristic $2$ has some nonzero $[x,x]$ with $|x|$
even, it cannot be faithfully represented by commutators, and thus
does not inject into its universal enveloping algebra. This
substantially changes portions of the theory of Lie algebras that
require an embedding into the universal enveloping algebra and so at
least one other definition of Lie algebra in characteristic $2$ is
sometimes used, one that forces an injection into a Lie algebra of
commutators.


In the case of $R = \F_2$ this other kind of Lie algebra simply adds
the requirement that $[x,x]=0$ for all homogeneous $x$. We will call
this kind of Lie algebra a $\Lie^\ti$-algebra ---the $\ti$ stands for
\emph{totally isotropic}. A definition for all rings $R$, due to
Moore, just forces the representation as a commutator Lie algebra to
exist:

\begin{definition}
  A graded $\Lie^\ti$-algebra (resp. $\sLie^\ti$-algebra) over $R$ is
  graded $R$-module $L$ with a bracket $L_i \otimes L_j \to L_{i+j}$
  (resp. $L_{i+j-1}$) and a monomorphism $L \to A$ to some graded
  associative algebra so that the bracket goes to the graded
  commutator $xy-(-1)^{|x||y|}yx$ (resp. $xy+(-1)^{|x||y|}yx$).
\end{definition}

Our main interest in these algebras is that the \emph{basic products}
appearing in Hilton's theorem about the loop space of a wedge of
spheres \cite{Hilton} form a basis (called a Hall basis) for a totally
isotropic Lie algebra, see the discussion in section \ref{s:homfree}.

\section{Homology operations for spectral Lie algebras}

Throughout this section $L$ will denote an algebra for the operad
$\idop$. So in particular, $L$ is a spectrum equipped with structure
maps $\xi_n : \D_n(L) \to L$ where
$\D_n(L) = \left( \idop[n] \sm L^{\sm n} \right)_{h\Sigma_n}$. 
There is a more traditional way to describe the structure of an
algebra for an operad: by giving maps
$\alpha_n : \idop[n] \sm L^{\sm n} \to L$ that are
$\Sigma_n$-equivariant for the trivial action on the codomain and the
diagonal action on the domain. The relation between these two styles
of definition is captured in the following commutative diagram:
\[\xymatrix{ \idop[n] \sm L^{\sm n} \ar[rr]^-{\alpha_n} \ar[d] && L
  \ar[d] \ar[dr]^-{=}
  & \\
  \left( \idop[n] \sm L^{\sm n} \right)_{h\Sigma_n}
  \ar@/_2pc/[rrr]_{\xi_n} \ar[rr]^-{(\alpha_n)_{h\Sigma_n}} && L \sm
  \Sigma^\infty_+ B \Sigma_n \ar[r] & L,}\]
where the vertical maps are the canonical maps
$Y^{\sm n} \to \sym[n]{Y}$ and the unlabeled horizontal map is
$L \sm \Sigma^\infty_+(-)$ applied to $B\Sigma_n \to \ast$.

\bigskip

In this section we will describe some operations on $H_\ast(L; \F_2)$
that will turn out to generate all others, and whose definition will
only require the map
\[\xi = \xi_2 : \left( \idop[2] \sm L^{\sm 2} \right)_{h\Sigma_2} \to L.\]
Recall that $\idop[2]$ is the Spanier--Whitehead dual of the partition
complex $P_2$. A set with two elements has only two partitions, both
fixed by $\Sigma_2$, so that $P_2 \simeq S^1$ and
$\idop[2] \simeq S^{-1}$, both with trivial $\Sigma_2$-action. This
implies that
$\left( \idop[2] \sm L^{\sm 2} \right)_{h\Sigma_2} \simeq \Sigma^{-1}
\sym{L}$.

\subsection{The shifted Lie bracket}

We will start by describing the Lie bracket. Here we remind the reader
that $\idop$ is not really analogous to the Lie operad, but rather is
analogous to its operadic suspension.

\begin{definition}
  The \emph{shifted Lie bracket} on the homology of an $\idop$-algebra
  $L$ is the map
  $[\cdot, \cdot] : H_i(L) \otimes H_j(L) \to H_{i+j-1}(L)$ given by
  the identification $S^{-1} \to \idop[2]$, that is, it is the map
  induced on homology by the suspension of the structure map
  $\alpha_2 : \Sigma^{-1} L^{\sm 2} \to L$.
\end{definition}

This operation really gives a $\sLie$-algebra:

\begin{proposition}\label{p:jacobi}
  Given any $\idop$-algebra $L$, the shifted Lie bracket on
  $H_\ast(L)$ gives $H_\ast(L)$ the structure of a
  $\sLie$-algebra.
\end{proposition}

\begin{remark}
  The proof of this proposition presented below is joint work with
  Lukas Brantner.
\end{remark}

\begin{proof}
  This is not really a result about $\F_2$-homology: we will show that
  the Lie algebra structure is present already at the level of
  spectra. We have already proved symmetry, when we computed
  $\idop[2]$ and saw it had the trivial $\Sigma_2$-action.

  Now we must prove the Jacobi identity. Consider three elements of
  the homology of $L$, say $x\in H_iL$, $y \in H_jL$ and $z \in H_kL$.
  Just as we defined the bracket $[x,y]$ as coming from a particular
  map $S^{-1} \to \idop[2]$, the bracket $[x,[y,z]] \in H_{i+j+k-2}L$
  can be obtained from a particular map $\nu : S^{-2} \to \idop[3]$ as
  the effect on homology of the double suspension of the composite
  \[S^{-2} \sm L^{\sm 3} \xrightarrow{\nu \sm \id} \idop[3] \sm L^{\sm
      3} \xrightarrow{\alpha_3} L.\]  
  The map $\nu$ simply corresponds to one of the structure maps of the
  $\idop$-operad under the identifications $S \simeq \idop[1]$,
  $S^{-1} \simeq \idop[2]$, namely, it is
  \[S^{-1} \sm (S \sm S^{-1}) \simeq \idop[2] \sm (\idop[1] \sm
    \idop[2]) \to \idop[1+2].\]
  
  The other terms in the Jacobi identity are obtained from $[x,[y,z]]$
  by cyclically permuting the variables, so, letting
  $\sigma = (123) \in \Sigma_3$ and using $\sigma_\ast$ to denote the
  induced action of $\sigma$ on $\idop[3]$, to prove the Jacobi
  identity we must show
  $\nu + \sigma_\ast \circ \nu + \sigma^2_\ast \circ \nu : S^{-2} \to
  \idop[3]$ is null-homotopic.

  We might as well work with $\Sigma^\infty P_3$, before taking
  Spanier--Whitehead duals, and show that
  $1 + \bar{\sigma} + \bar{\sigma}^2 : \Sigma^\infty P_3 \to
  \Sigma^\infty P_3$ is null-homotopic, where $\bar{\sigma}$ denotes
  the action of the 3-cycle $\sigma$ on $\Sigma^\infty P_3$.

  Now, $P_3$ consists of:
  \begin{itemize}
  \item a $1$-simplex, corresponding to the chain $\hat{0} < \hat{1}$,
    connecting the basepoint $\hat{0} = \hat{1}$ with itself, and
  \item three $2$-simplices, say $\tau_1$, $\tau_2$, $\tau_3$, each
    filling in the above circle, corresponding to the three chains
    $\hat{0} < (23|1) < \hat{1}$, $\hat{0} < (13|2) < \hat{1}$, and
    $\hat{0} < (12|3) < \hat{1}$, respectively.
  \end{itemize}

  The $3$-cycle $\sigma$ permutes those three $2$-simplices
  cyclically. We can compute
  $1 + \bar{\sigma} + \bar{\sigma}^2 : \Sigma^\infty P_3 \to
  \Sigma^\infty P_3$ as the composite:
  \[ \Sigma^\infty P_3 \xrightarrow{\Delta} \bigvee^3 \Sigma^\infty
    P_3 \xrightarrow{ 1 \vee \bar{\sigma} \vee \bar{\sigma}^2 }
    \bigvee^3 \Sigma^\infty P_3 \xrightarrow{\nabla} \Sigma^\infty
    P_3. \]

  Non-equivariantly we have an equivalence
  $S^2 \vee S^2 \xrightarrow{\simeq} P_3$, where we will think of the
  first $S^2$ as mapping to $P_3$ by sending the northern hemisphere
  to $\tau_1$, and the southern hemisphere to $\tau_2$; we will
  abbreviate this map $S^2 \to P_3$ as $\tau_{12}$ and use similar
  notation for other maps. We will think of the second wedge summand
  $S^2$ as corresponding to the map $\tau_{23}$.
  
  We can think of a map
  $\bigvee^n \Sigma^\infty P_3 \to \bigvee^m \Sigma^\infty P_3$ as
  given by an $m \times n$ matrix of maps
  $\Sigma^\infty P_3 \to \Sigma^\infty P_3$, and each such map as
  given by a $2 \times 2$ matrix of maps
  $\Sigma^\infty S^2 \to \Sigma^\infty S^2$.

  The matrices of $\Delta$ and $\nabla$ are just the $3 \times 1$ and
  $1 \times 3$ matrices each of whose entries is $I$, the $2 \times 2$
  identity matrix. Once we have the $2 \times 2$ matrix $A$
  representing $\sigma : \Sigma^\infty P_3 \to \Sigma^\infty P_3$, the
  matrix of $1 \vee \bar{\sigma} \vee \bar{\sigma}^2$ is given by the
  $3 \times 3$ diagonal matrix with $I, A, A^2$ along the diagonal.
  
  To compute the matrix $A$, notice that
  $\tilde\sigma \circ \tau_{12} = \tau_{23}$ and
  $\tilde\sigma \circ \tau_{23} = \tau_{31}$ (where $\tilde \sigma$
  denotes the action of $\sigma$ on $P_3$, so that
  $\Sigma^\infty \tilde \sigma = \bar\sigma$). The map
  $\Sigma^\infty \tau_{13}$ is given by
  $\Sigma^\infty \tau_{12} + \Sigma^\infty \tau_{23}$, and $\tau_{31}$
  differs from $\tau_{13}$ by the reflection swapping the hemispheres
  of $S^2$, which has degree $-1$. So, \[A =
  \begin{pmatrix}
    0 & -1 \\ 1 & -1
  \end{pmatrix}.
  \]

  This means the composite map $1 + \bar{\sigma} + \bar{\sigma}^2$ has
  matrix:
  \[
    \begin{pmatrix}
      I & I & I
    \end{pmatrix}
    \begin{pmatrix}
      I & 0 & 0 \\
      0 & A & 0 \\
      0 & 0 & A^2 
    \end{pmatrix}
    \begin{pmatrix}
      I \\ I \\ I
    \end{pmatrix}
    = I + A + A^2,
   \]
   which is readily computed to be $0$.
\end{proof}

\subsection{Behrens's unary Dyer-Lashof-like operations}

In \cite[Chapter 1]{Behrens}, Behrens interprets Arone and Mahowald's
calculation \cite{AroneMahowald} of $H_\ast(\D_n(X))$ for a sphere $X$
in the case $p=2$ in terms of unary homology operations for the layer
of the Goodwillie tower of a reduced finitary homotopy functor
$F : \Topp \to \Topp$. The Arone-Ching chain rule \cite{AroneChing}
gives the symmetric sequence of derivatives of $F$,
$\partial_\ast(F)$, the structure of a bimodule for $\idop$. Behrens's
operations only use the left module structure and could be defined on
the mod $2$ homology of any symmetric sequence which is a left module
over $\idop$. In particular, regarding an $\idop$-algebra as a
symmetric sequence concentrated in degree $0$, we get unary operations
on the mod $2$ homology of an $\idop$-algebra:

\begin{definition}[{adapted from \cite[Section 1.5]{Behrens}}]\label{d:unops}
  Let $L$ be a spectrum equipped with the structure of an
  $\idop$-algebra. We define homology operations
  \[ \Q^j : H_d(L) \to H_{d+j-1}(L),\]
  as follows: for $x \in H_d(L)$, we set
  $\Q^j x := \xi_\ast \sigma^{-1} Q^j x$ where
  \begin{itemize}
  \item $\xi : \Sigma^{-1} \sym{L} \xrightarrow{\simeq} \D_2(L) \to L$
    is part of the $\idop$-algebra structure of $L$,
  \item $\sigma^{-1} : H_{d+j}(\sym{L}) \to H_{d+j-1}(\D_2(L))$ is the
    (de)suspension isomorphism, and
  \item $Q^j : H_d(L) \to H_{d+j}(\sym{L})$ is a Dyer-Lashof
    operation.
  \end{itemize}
\end{definition}

Note that $\Q^j$ has degree $j-1$ but the notation for it uses ``$j$''
because it is named after $Q^j$. Also notice that if $j<d$ and
$x \in H_d(L)$, we have $\Q^j x = 0$ simply because $Q^j x = 0$.

\begin{remark}\label{r:agree}
  By modifying the setting of the definition of the $\Q^j$, we have
  introduced a potential ambiguity! For a free $\idop$-algebra
  $L = F_{\idop}(X)$ on some spectrum $X$, there are two different
  ways in which we could mean $\Q^j x$ for $x \in H_\ast(L)$: using
  definition \ref{d:unops}, or using Behrens' original definition for
  the functor $\Id$. Let us explain what that definition is and show
  it agrees with our definition in this case.

  Given a functor $F : \Topp \to \Topp$, part of the left
  $\idop$-module structure on $\partial_\ast(F)$ is a
  $\Sigma_2 \wr \Sigma_i$-equivariant map
  $\idop[2] \sm \partial_i(F)^{\sm 2} \to \partial_{2i}(F)$. This
  induces a map
  \begin{align*}
    \psi_i : \Sigma^{-1}\sym{\left( \D_i(F)(X) \right)}
    & \simeq \left(\idop[2] \sm \partial_i(F)^{\sm 2} \sm X^{\sm 2i}\right)_{h \Sigma_2 \wr \Sigma_i} \\
    & \to \left(\partial_{2i}(F) \sm X^{\sm 2i}\right)_{h \Sigma_{2i}} \\
    & \simeq \D_{2i}(F)(X),
  \end{align*}
  and for $x \in H_d(\D_i(F)(X))$, Behrens defines
  $\Q^j x = (\psi_i)_\ast \sigma^{-1} Q^j x$.

  Given $x \in H_d(\D_i(\Id)(X)) \subset H_\ast(F_{\idop}(X))$ and a
  $j \ge d$, to show that the $\Q^j x$ from Definition \ref{d:unops}
  agrees with this original version of
  $\Q^j x \in H_{d+j}(\D_{2i}(\Id)(X)) \subset H_\ast(F_{\idop}(X))$
  we just need to unwind the definitions, the point being that both
  the left $\idop$-module structure of $\idop$ and the $\idop$-algebra
  structure of $F_{\idop}(X)$ come directly from the operad structure
  maps of $\idop$.
\end{remark}

\begin{definition}\label{d:qsrel}
  Let $\Qs$ be the $\F_2$-algebra freely generated by symbols
  $\{\Q^j : j \ge 0\}$ subject to the following relations:
  \[ \Q^r \Q^s = \sum_{k=0}^{r-s-1} \binom{2s-r+1+2k}{k} \Q^{2s+1+k}
    \Q^{r-s-1-k}, \quad \text{if } s < r\le 2s. \]
  Also let $\Qs_n$ be the quotient of $\Qs$ obtained by imposing the
  additional relations $\Q^{j_1} \Q^{j_2} \cdots \Q^{j_k} = 0$
  whenever $j_1 < j_2 + \cdots + j_k + n$.
\end{definition}

The relations in $\Qs_n$ allow one to rewrite any monomial in the
$\Q^j$ into a linear combination of \emph{completely unadmissible} or
\emph{CU} monomials, that is, monomials
$\Q^J = \Q^{j_1} \Q^{j_2} \cdots \Q^{j_k}$ where
$J = (j_1, \ldots, j_k)$ is a (possibly empty, corresponding to
$1 \in \Qs$) sequence of integers satisfying $j_i > 2j_{i+1}$ for
$i=1, \ldots, k-1$.

\begin{definition}\label{d:allowmod}
  A positively graded module $M$ over $\Qs$ is called \emph{allowable}
  if whenever $x \in M$ is homogeneous of degree $n$ and
  $j_1 < j_2 + \cdots + j_k + n$, we have
  $\Q^{j_1} \Q^{j_2} \cdots \Q^{j_k} x = 0$.
\end{definition}

\begin{remark}
  This notion of allowable requires more operations to vanish than
  required by degree considerations, that is, more than required by
  the condition $\Q^j x = 0$ when $x \in M_n$, $j<n$. Indeed, that
  last condition only implies
  $\Q^{j_1} \Q^{j_2} \cdots \Q^{j_k} x = 0$ when
  $j_i < j_{i+1} + \cdots j_k + n - (k-i)$ for some $i$; note the
  extra negative term $-(k-i)$. The reason for this extra vanishing
  required is the isomorphism in \cite[Theorem 1.5.1]{Behrens}, that
  in the notation used there, sends
  $\sigma^k \Q^{j_1} \Q^{j_2} \cdots \Q^{j_k} \iota_n \mapsto Q^{j_1}
  Q^{j_2} \cdots Q^{j_k} \iota_n$. The $Q^j$ do have that vanishing
  property just for degree reasons.
\end{remark}

\begin{proposition}\label{p:unaries}
  Given an $\idop$-algebra $L$, the action of the operations $\Q^j$
  makes $H_{> 0}(L)$ into an allowable $\Qs$-module.
\end{proposition}

\begin{proof}
  We will deduce that the operations act allowably and satisfy the
  relations in the algebra $\Qs$ from \cite[Theorem 1.5.1]{Behrens}.
  That theorem states that
  \[ \bigoplus_{k \ge 0} H_\ast(\D_{2^k}(S^n)) = \Qs_n \{ \iota_n
    \},\] where $\iota_n$ is the fundamental class of
  $\tilde{H}_n(S^n)$ (thought of as living in
  $H_n(\D_1(S^n)) \cong \tilde{H}_n(S^n)$), and the operations $\Q^j$
  obey all the relations in the algebra $\Qs_n$.

  Given any class $x \in H_n(L)$, we can represent it by map
  $x : \HF[n] \to \HF \sm L$ of $\HF$-module spectra. This corresponds
  to a map $x^\dagger : \HF \sm F_{\idop}(S^n) \to \HF \sm L$ of
  $(\HF \sm \idop)$-algebras. The naturality of the $\Q^j$ operations
  shows that given any $R \in \Qs$ we have
  $H_\ast(x^\dagger) R\iota_n = Rx$, so that if the relation
  $R\iota_n = 0$ is satisfied in $H_\ast(F_{\idop}(S^n))$, the
  relation $R x = 0$ holds in $H_\ast(L)$.
\end{proof}

Notice that it also follows from theorem \cite[Theorem
1.5.1]{Behrens}, that the \emph{CU}-monomials $\Q^J$ are linearly
independent.

\section{Algebraic structure of homology of spectral Lie algebras}
\label{s:algstruct}

We can now state the algebraic structure of the homology of an $\idop$-algebra:

\begin{definition}\label{d:allowlie}
  An \emph{allowable $\Qs$-$\sLie$-algebra} is a graded $\F_2$-vector
  space $M$, equipped with
  \begin{itemize}
  \item a shifted Lie bracket $[-,-] : M_i \otimes M_j \to M_{i+j-1}$,
    and
  \item the structure of an allowable $\Qs$-module on $M_{> 0}$,
  \end{itemize}
  such that
  \begin{enumerate}
  \item\label{d:square} $\Q^k x = [x,x]$ if $x \in M_k$, and
  \item\label{d:vanishing} $[x,\Q^k y] = 0$ for any $x \in M_i$,
    $y \in M_j$.
  \end{enumerate}
\end{definition}

\begin{remark}
  Notice that condition \ref{d:vanishing} only has content when
  $k \ge j$, since otherwise $\Q^k y = 0$.
\end{remark}

\begin{theorem}\label{t:hstar-allowlie}
  Given any $\idop$-algebra $L$, the operations described above give
  its mod $2$ homology $H_\ast(L)$ the structure of an allowable
  $\Qs$-$\sLie$-algebra.
\end{theorem}

\begin{proof}
  We have already shown that the bracket gives $H_\ast(L)$ the
  structure of a $\sLie$-algebra and of an allowable $\Qs$-module in
  Propositions \ref{p:jacobi} and \ref{p:unaries}. We will prove
  properties 1 and 2 from Definition \ref{d:allowlie} in Lemmas
  \ref{l:square} and \ref{l:vanishing} below.
\end{proof}

It will be convenient to recall a construction of the Dyer-Lashof
operation $Q^k : H_j(L) \to H_{j+k}(\sym{L})$ for $k \ge j$. A class
$x \in H_j(L)$ can be represented by a map $x : \HF[j] \to \HF \sm L$
of $\HF$-module spectra. Applying the second extended power functor we
get a map $\symt{x} : \symt{(\HF[j])} \to \symt{(\HF \sm L)}$, where
we have used $\tp$ for the smash product of $\HF$-module spectra. Since
the free $\HF$-module functor is symmetric monoidal and preserves
homotopy colimits, $\symt{(\HF \sm Y)} \simeq \HF \sm \sym{Y}$;
so that we can regard $\symt{x}$ as being a map
$\symt{(\HF[j])} \to \HF \sm \sym{L}$.

Now, $(\HF[j])^{\tp\,2}$ has trivial $\Sigma_2$-action, so
$\symt{(\HF[j])} \simeq \HF[2j] \sm \Sigma^\infty_+ B\Sigma_2$. One way
to see this is to recall that the homotopy category of $\HF$-module
spectra is equivalent to the derived category of complexes of
$\F_2$-vector spaces, and the definition of the symmetry of the tensor
product has no signs in that case. Alternatively, one can think of
spaces, before smashing with $\HF$: the fibration $(S^j)^{\sm 2} \to
\sym{S^j} \to B\Sigma_2$ is $\F_2$-orientable, and the above
equivalence is an instance of the Thom isomorphism.

Let $q_{k-j} : \HF[k-j] \to \HF \sm \Sigma^\infty_+ B\Sigma_2$ pick
out the unique non-zero class of degree $k-j$ in $H_\ast(B\Sigma_2)$;
then $Q^k x$ is represented by
\begin{align*}
\HF[j+k] \xrightarrow{q_{k-j} \tp \id_{\HF[2j]}} & (\HF \sm
  \Sigma^\infty_+ B\Sigma_2) \tp \HF[2j] \\
  & \simeq \symt{(\HF[j])}
  \xrightarrow{\symt{x}} \HF \sm \sym{L}.
\end{align*}

\begin{lemma}\label{l:square}
  For any $\idop$-algebra $L$ and $x \in H_k(L)$, we have
  $\Q^k x = [x,x]$.
\end{lemma}

\begin{proof}
  This follows easily by unwinding the definitions: if $x$ is
  represented by a map $x : \HF[k] \to \HF \sm L$, both sides are
  represented by the desuspension of some composite
  \[ \HF[k] \tp \HF[k] \to \symt{\left(\HF[k] \right)}
  \xrightarrow{\symt{x}} \HF \sm \sym{L} \xrightarrow{\HF\sm\Sigma\xi}
  \HF \sm L, \]
  where $\xi : \Sigma^{-1} \sym{L} \to L$ is the structure map. For
  $[x,x]$ the first map is taken to be the quotient map, while for
  $\Q^k x$ it is $q_0 \tp \id_{\HF[2k]}$, which agrees with the
  quotient map.
\end{proof}

\begin{lemma}\label{l:vanishing}
  For an $\idop$-algebra $L$ and $x \in H_i(L)$, $y \in H_j(L)$ we
  have $[x,\Q^k y] = 0$.
\end{lemma}

\begin{proof}
  For $k<j$, $\Q^k y = 0$. For $k=j$, by Lemma \ref{l:square},
  $[x,\Q^k y] = [x,[y,y]]$ and this is $0$ as explained in section
  \ref{s:liealg}.

  To analyze the case $k>j$, we begin by unwinding the definitions in
  terms of representing maps $x : \HF[i] \to \HF \sm L$ and
  $y : \HF[j] \to \HF \sm L$. To make the next diagram fit on the
  page, we introduce some temporary notation: $[i] := \HF[i]$,
  $\bar{L} := \HF \sm L$,
  $\overline{B\Sigma_2} := \HF \sm \Sigma^\infty_+ B \Sigma_2$ and
  $\partial_n := \idop[n]$. Then $[x,\Q^k y] \in H_{i+j+k-2}(L)$ is
  represented by the the composite from the top left corner to the
  bottom right corner in the following commutative diagram:
  \[
    \xymatrix{
      [i + j + k - 2]
      \ar[d]^{\id_{[i-1]} \tp \Sigma^{-1}q_{k-j} \tp \id_{[2j]}} & \\
      \Sigma^{-1}[i] \tp \Sigma^{-1}(\overline{B \Sigma_2} \tp [2j])
      \ar[d]_\simeq & \\
      \partial_2 \sm (\partial_1 \sm [i]) \tp \symt{(\partial_2 \sm
        [j])}
      \ar[d]_{(\partial_1 \sm x) \tp \symt{(\partial_2 \sm y)}} \ar[r]^-{\theta_{[i],[j]}} &
      (\partial_3 \sm [i+2j])_{h(\Sigma_1 \times \Sigma_2)}
      \ar[d]^{(\partial_3 \sm x \tp y^{\tp 2})_{h(\Sigma_1 \times \Sigma_2)}}
      \\
      \partial_2 \sm (\partial_1 \sm \bar{L}) \tp \symt{(\partial_2
        \sm \bar{L})}
      \ar[d]_{\id\tp(\HF\sm\xi)} \ar[r]^-{\theta_{\bar{L},\bar{L}}} &
      (\partial_3 \sm \bar{L}^{\sm 3})_{h(\Sigma_1 \times \Sigma_2)}
      \ar[d]^{\xi'_3}
      \\
      \partial_2 \sm \bar{L} \tp \bar{L}
      \ar[r]_-{\HF\sm\alpha_2} &
      \bar{L}.
    }
  \]
  The horizontal arrows whose labels involve $\theta$ are defined
  using the structure map
  $\theta : \partial_2 \sm \partial_1 \sm \partial_2 \to \partial_3$,
  namely,
  \[\theta_{X,Y} : \partial_2 \sm (\partial_1 \sm X) \tp
    \symt{(\partial_2 \sm Y)} \to (\partial_3 \sm X \tp Y^{\tp
      2})_{h(\Sigma_1 \times \Sigma_2)}\]
  is given by
  $(\theta \sm \id_{X \tp Y^{\tp 2}})_{h(\Sigma_1 \times \Sigma_2)}$.    

  The arrow labeled $\xi'_3$ is $\HF$ smashed with the composite
  \[ \left( \partial_3 \sm L^{\sm 3} \right)_{h(\Sigma_1 \times
      \Sigma_2)} \xrightarrow{(\alpha_3)_{h(\Sigma_1 \times
        \Sigma_2)}} L \sm \Sigma^\infty_+ B \Sigma_2 \to L,\]
    and that the bottom square commutes follows from the definition
    of algebra for an operad.

    To conclude the proof, we will show that
    $(\partial_3 \sm [i+2j])_{h(\Sigma_1 \times \Sigma_2)}$ is
    concentrated in degree $i+2j-2$, which means the composite from
    the top of the diagram to that point must be null if $k \neq j$.
    Now, that spectrum is equivalent to
    $\HF \sm \Sigma^{i+2j} (\partial_3)_{h(\Sigma_1 \times \Sigma_2)}$
    because the $(\Sigma_1 \times \Sigma_2)$-action on $[i+2j]$ is
    trivial. So we need to describe $\partial_3$ as a
    $(\Sigma_1 \times \Sigma_2)$-spectrum. Recall the description of
    $P_3$ from the proof of Proposition \ref{p:jacobi}: it consists of
    three $2$-dimensional disks with their boundaries identified, one
    for each of the three partitions $(12|3)$, $(13|2)$, $(23|1)$. The
    $(\Sigma_1 \times \Sigma_2)$-action fixes one of the disks and
    swaps the other two, so that $P_3$ is equivariantly equivalent to
    $\Sigma^2 \Sigma^\infty_+ \Sigma_2$, the double suspension of the
    regular representation of $\Sigma_2$. Then $\partial_3$ is
    $\Sigma^{-2} \Sigma^\infty_+ \Sigma_2$ and
    $(\partial_3)_{h(\Sigma_1 \times \Sigma_2)} \simeq S^{-2}$, as
    required.
\end{proof}

\begin{remark}\label{r:lukas}
  Lukas Brantner has written an alternative proof of this lemma that
  will appear in \cite{Lukas}. His argument analyzes the structure map
  $\theta$ showing it is the double desuspension of the transfer map
  $\Sigma^\infty_+ B\Sigma_2 \to S$ and thus vanishes on mod $2$
  homology. I am grateful to him for sharing his proof with me at a
  time when I was still confused about the ``bottom operation'' and
  thought this result only held for $k>j$.
\end{remark}

\section{Homology of free spectral Lie algebras on simply-connected spaces}
\label{s:homfree}

Now we can state our main result:

\begin{theorem}\label{t:homfree}
  Given a simply-connected space $X$, the mod $2$ homology of the free
  $\idop$-algebra on $\Sigma^\infty X$ is the free allowable
  $\Qs$-$\sLie$-algebra $\frQsLie(\tilde{H}_\ast(X))$ on the reduced
  homology $\tilde{H}_\ast(X)$.

  More precisely, the canonical map
  $\frQsLie(\tilde{H}_\ast(X)) \to H_\ast(F_{\idop}(\Sigma^\infty
  X))$ is an isomorphism.
\end{theorem}

\begin{remark}
  The restriction to suspension spectra of simply-connected spaces is a
  consequence of the use of the Hilton--Milnor theorem in section
  \ref{s:wedge}; it would be interesting to work out more general
  cases, and that seems to be the main obstacle. The restrictions to
  positive degrees in Definition \ref{d:allowmod} and Proposition
  \ref{p:unaries} are not essential, as the homology of
  $\bigoplus_{k \ge 0} H_\ast(\D_{2^k}(S^n))$ can also be calculated
  for negative $n$. Indeed, since $\HF \sm \sym[m]{\Sigma^j S^n}
  \simeq \Sigma^{jm} (\HF \sm \sym[m]{S^n})$ (by the same arguments
  used near the end of the proof of Theorem \ref{t:hstar-allowlie}:
  the symmetry of the tensor product of $\HF$-spectra has no signs, or
  a Thom isomorphism), one can easily show deduce the homology for
  negative spheres from the Arone--Mahowald calculation. We have left
  the restrictions to positive degrees \emph{there} for simplicity,
  since we do not know how to avoid them in the main result anyway.
\end{remark}

We will prove Theorem \ref{t:homfree} in special cases of increasing
generality in the next few sections, but first we will give a
convenient construction of the free allowable $\Qs$-$\sLie$-algebra.
This will involve the notion of \emph{basic products}, that we now recall:

\begin{definition}
  The \emph{basic products} on a set of \emph{letters}
  $x_1, \ldots, x_n$ are defined and ordered recursively as follows:
  
  The basic products of weight $1$ are $x_1, x_2, \ldots, x_n$,
  ordered arbitrarily. 

  Suppose the basic products of weight less than $k$ have been defined
  and ordered. A basic product of weight $k$ is a formal bracket
  $\fb{w_1, w_2}$ where
 \begin{itemize}
 \item $w_1$ and $w_2$ are basic products whose weights add up to $k$,
 \item $w_1<w_2$ in the order defined so far,
 \item if $w_2 = \fb{w_3, w_4}$ for some basic products
   $w_3$ and $w_4$, then we require that $w_3 \le w_1$.
 \end{itemize}
 Once all the products of weight $k$ are defined, they are ordered
 arbitrarily among themselves and declared to be greater that all
 basic products of lower weight. We will assume these choices of order
 are fixed once and for all.
\end{definition}

Marshall Hall proved in \cite{HallBasis} that the basic products form
a basis for the free Lie algebra on $x_1, x_2, \ldots, x_n$. That
result is for the totally isotropic, ungraded version of Lie algebra,
but it clearly extends, at least for $R = \F_2$ where the grading does
not introduce signs, to both $\Lie^\ti$-algebras and
$\sLie^\ti$-algebras: if the letters have assigned degrees $|x_i|$, we
assign to each basic product $w$ with $\ell$ letters of total degree
$d$, the degree $|w|=d$ in the $\Lie^\ti$ case and $|w| = d-\ell$ in
the $\sLie^\ti$ case.

\begin{proposition}\label{p:frQsLie}
  The free allowable $\Qs$-$\sLie$-algebra $\frQsLie(V)$ on a graded
  $\F_2$-vector space $V$ is the free allowable $\Qs$-module on the
  free $\sLie^\ti$ algebra on $V$, denoted by
  $\aQs(F_{\sLie^\ti}(V))$, equipped with a bracket defined as
  follows:

  First, fix a basis $\beta$ of $V$ and consider the basis of
  $\aQs(F_{\sLie^\ti}(V))$ consisting of all $\Q^J w$ where:
  \begin{itemize}
  \item $J = (j_1, \ldots, j_k)$ is a $CU$-sequence of integers, and
  \item $w$ is a basic product of degree at most $j_k$ in letters
    from $\beta$.
  \end{itemize}
  Now define the bracket on $\aQs(F_{\sLie^\ti}(V))$ on that basis as
  indicated below and extended bilinearly:
  \begin{itemize}
  \item $[\Q^{J_1} w_1, \Q^{J_2} w_2] = 0$ if $J_1 \neq \emptyset$ or $J_2 \neq \emptyset$.
  \item The bracket $[w_1,w_2]$ of basic products is defined
    recursively as follows:
    \begin{enumerate}
    \item If $\fb{w_1, w_2}$ is a basic product, we let
      $[w_1,w_2] = \fb{w_1, w_2}$.
    \item\label{i:square} $[w_1,w_2] = \Q^{|w_1|} w_1$ if $w_1 = w_2$.
    \item $[w_1,w_2] = [w_2,w_1]$  if $w_1 > w_2$.
      
    \item $[w_1,w_2] = [w_3,[w_1,w_4]] + [w_4,[w_1,w_3]]$  if $w_1 < w_2$ and $w_2=[w_3,w_4]$ with $w_1<w_3$.
    \end{enumerate}
  \end{itemize}
\end{proposition}

\begin{proof}
  In \cite{HallBasis}, Hall defines the $\Lie^\ti$ bracket on the
  linear span of the basic products as above, except that
  (\ref{i:square}) is replaced with $[w_1,w_1]=0$. He then proves that
  the recursion in the definition does terminate and that it produces
  a $\Lie^\ti$-algebra, that is, that the bracket is anti-symmetric,
  satisfies the Jacobi identity and $[x,x]=0$ for all $x$. A
  straightforward adaptation of his proof will show that the above
  definition also terminates and produces an allowable
  $\Qs$-$\sLie$-algebra. But before we explain that, let us assume the
  bracket does define a $\Qs$-$\sLie$-algebra and check that it is
  free. Let $f : V \to E$ be a morphism of graded vector spaces where
  $E$ is an allowable $\Qs$-$\sLie$-algebra. There is a unique
  bracket-preserving extension of $f$ to the linear span of the basic
  products, and therefore a unique extension of $f$ to a morphism of
  allowable $\Qs$-modules $\aQs(F_{\sLie^\ti}(V)) \to E$. That this
  unique extension is also a morphism of allowable
  $\Qs$-$\sLie$-algebras is clear from the above definition of the
  bracket.

  \smallskip

  And now we check the bracket correctly produces an allowable
  $\Qs$-$\sLie$-algebra. First of all, notice that the degrees of the
  various parts of the definition are correct for a shifted bracket.

  Secondly, having $[w_1,w_1] = \Q^{|w_1|} w_1$ instead of $0$ does
  not affect termination of the recursion at all. Both $0$ and
  $\Q^{|w_1|} w_1$ have the following properties: (1) they are
  expressions containing no further brackets, so if a term reduces to
  one of them that term requires no further reduction, and (2) if they
  appear \emph{inside} a bracket, the term containing that bracket is
  $0$. This means that the process of reducing a bracket $[x,y]$ to a
  linear combination of basic products by repeatedly applying the
  recursive definition uses exactly the same steps in both Hall's
  $\Lie^\ti$ case and in our $\Qs$-$\sLie$ case, the only difference
  being that any $[w,w]$ that appear on their own (that is, not inside
  a bracket) will reduce to $\Q^{|w|} w$ instead of $0$.
  
  Next we must check that this bracket satisfies $[x,\Q^ky] = 0$,
  $[x,x]=\Q^{|x|} x$, symmetry and the Jacobi identity. All of these
  need only be checked on the given basis. Symmetry and that
  $[x,\Q^k y] = 0$ are directly built in to the definition, as is the
  fact that $[x,x] = \Q^{|x|} x$ when $x$ is a basic product. When
  $x = \Q^J w$ for $J = (j_1, \ldots, j_k)$ with $k \ge 1$, we have
  $[x,x] = 0$ (since $J \neq \emptyset$), but we also have
  $|x| = j_1 + \cdots + j_k + |w| - k < j_1 + \cdots + j_k + |w|$ so
  that $\Q^{|x|} x = \Q^{|x|} \Q^J w = 0$ is required by allowability.

  Now only the Jacobi identity remains to be checked:
  \[ \sum_{\mathrm{cyclic}} \left[\Q^{J_1}w_1, [\Q^{J_2}w_2,
      \Q^{J_3}w_3]\right] = 0. \]
    If any $J_i \neq \emptyset$, all three terms are $0$, so assume
    all $J_i = \emptyset$. This remaining case can be proved exactly
    as in \cite[Section 3, p. 579]{HallBasis}, with one tiny change.
    There is only one place in that proof where the condition
    $[w,w]=0$ is used: it is at the very beginning of the argument for
    the Jacobi identity. The proof starts by considering the case when
    two of the $w_i$ are equal, say $w_1=w_2$. Then the terms
    $[w_1,[w_1,w_3]]$ and $[w_1,[w_3,w_1]]$ cancel by anti-symmetry
    and the remaining term is $0$ since $[w_3,[w_1,w_1]] = [w_3, 0]$.
    In our case, that last term still vanishes:
    $[w_3,[w_1,w_1]] = [w_3, \Q^{|w_1|} w_1] = 0$. The rest of Hall's
    argument goes through verbatim.
\end{proof}

\subsection{The free spectral Lie algebra on a sphere}

For $X = S^n$, Theorem \ref{t:homfree} is essentially a restatement of
\cite[Theorem 1.5.1]{Behrens} using Proposition \ref{p:frQsLie}.
Indeed, the free $\sLie^\ti$-algebra on
$\tilde{H}_\ast(S^n) = \F_2 \{ \iota_n \}$ is just $\F_2\{\iota_n\}$
again, so that
$\frQsLie(\tilde{H}_\ast(S^n)) = \aQs(\F_2\{\iota_n\})$, which is what
Behrens shows $H_\ast(F_{\idop}(S^n))$ to be.

\subsection{The free spectral Lie algebra on a finite wedge of spheres}
\label{s:wedge}

Now we consider the case of
$X = S^{d_1} \vee S^{d_2} \vee \cdots \vee S^{d_k}$ for some integers
$d_i \ge 2$; in this case $F_{\idop}(X)$ can be computed from the
results of \cite{AroneKankaanrinta}, which we now summarize.

\medskip 

Consider a bit more generally the case
$X = \Sigma(X_1 \vee \cdots \vee X_k)$, where the $X_i$ are some
connected spaces. In \cite{AroneKankaanrinta} there is a computation
of
$\D_n(\Id)(X) = \Sigma \D_n(\Omega\Sigma)(X_1 \vee \cdots \vee X_k)$
that ``takes multi-variable Goodwillie derivatives on both sides of the
Hilton--Milnor theorem''.

The Hilton--Milnor theorem (see \cite[Section XI.6]{Whitehead}) gives a
homotopy equivalence between
$\Omega \Sigma X = \Omega \Sigma (X_1 \vee \cdots \vee X_k)$ and the
weak\footnote{This means the homotopy colimit of the finite products,
  where the maps in the colimit include a product into a larger
  product using the basepoint on the extra factors.} infinite product
$\prod_w \Omega \Sigma Y_w(X_1, \ldots, X_k)$, where $w$ runs over the
basic products on $k$ letters, and each $Y_w$ is the functor obtained
from the word $w$ by interpreting the $i$-th letter as $X_i$, and the
bracket as the smash product; so that
$Y_w(X_1, \ldots, X_k) = X_1^{\sm m_1(w)} \sm \cdots \sm X_k^{\sm
  m_k(w)}$
with $m_i(w)$ counting the number of occurrences of the $i$-th letter
in $w$.

Given a basic product $w$ there is a map
$h_w : Y_w(X_1, \ldots, X_n) \to \Omega \Sigma X$ obtained from $w$ by
interpreting the $i$-th letter as the canonical map
$X_i \hookrightarrow X \to \Omega\Sigma X$ and interpreting the
bracket as the Samelson product. Let
$\bar{h}_w : \Omega \Sigma Y_w(X_1, \ldots, X_n) \to \Omega \Sigma X$
be the extension of $h_w$ to a map of $A_\infty$-spaces and for any
set $B$ of basic words let $\bar{h}_B$ be the composite
$ \prod_{w \in B} \Omega \Sigma Y_w(X_1, \ldots, X_n)
\xrightarrow{\prod \bar{h}_w} \left(\Omega \Sigma X\right)^{B}
\xrightarrow{\mu} \Omega \Sigma X$.
Then the Hilton--Milnor theorem can be stated as saying that the
colimit of $\bar{h}_B$ over all finite sets of basic products is an
equivalence.

\bigskip

The result Arone and Kankaanrinta obtain from the Hilton--Milnor
equivalence \cite[Theorem 0.1]{AroneKankaanrinta} is the following
equivalence of spectra:
\begin{align*}
  \bigl( \partial_n(\Omega \Sigma) \sm \Sigma^\infty\left(X_1^{\sm
  n_1} \sm \cdots \sm X_k^{\sm n_k}\right)\bigr)_{h(\Sigma_{n_1}
  \times \cdots \times \Sigma_{n_k})} \simeq \\ \bigvee_{d\mid\gcd(n_1,\ldots,n_k)}\left( \bigvee_{w \in
  W(\frac{n_1}{d}, \ldots, \frac{n_k}{d})} \D_d(\Omega\Sigma)\left(Y_w(X_1, \ldots  X_k)\right) \right),
\end{align*}

where $W(\frac{n_1}{d}, \ldots, \frac{n_k}{d})$ is the set of basic
products on $k$-letters involving the $i$-th letter exactly
$\frac{n_i}{d}$ times.

We can use this to get a nice formula for $\Sigma^{-1} F_{\idop}(X)$:
\begin{align*}
  & F_{\partial_\ast(\Omega\Sigma)}(X_1 \vee \cdots \vee X_k) = \bigvee_n \D_n(\Omega\Sigma)(X_1 \vee \cdots \vee X_k) \\
  = & \bigvee_n \bigl( 
      \partial_n(\Omega \Sigma) \sm \Sigma^\infty( X_1 \vee \cdots \vee  X_k)^{\sm n} \bigr)_{h\Sigma_n} \\
  = & \bigvee_n \left( \partial_n(\Omega \Sigma) \sm \bigvee_{n_1 + \cdots + n_k=n} \Ind^{\Sigma_n}_{\Sigma_{n_1} \times \cdots \times \Sigma_{n_k}} (X_1^{\sm n_1} \sm \cdots \sm X_k^{\sm n_k}) \right)_{h \Sigma_n} \\
  = & \bigvee_{n_1, \ldots, n_k} \left( \partial_{n_1 + \cdots + n_k}(\Omega \Sigma)
      \sm
      (X_1^{\sm n_1} \sm \cdots \sm  X_k^{\sm n_k}) \right)_{h (\Sigma_{n_1} \times \cdots \times \Sigma_{n_k})} \\
  = & \bigvee_{n_1, \ldots, n_k} \left( \bigvee_{d\mid\gcd(n_1,\ldots,n_k)}\left( \bigvee_{w \in
      W(\frac{n_1}{d}, \ldots, \frac{n_k}{d})} \D_d(\Omega\Sigma)\left( Y_w(X_1, \ldots, X_k) \right)\right)   \right) \\
  = & \bigvee_{m_1, \ldots, m_k, d} \left( \bigvee_{w \in
      W(m_1, \ldots, m_k)} \D_d(\Omega\Sigma)\left( Y_w(X_1, \ldots, X_k) \right)\right) \\
  = & \bigvee_{w \in W} F_{\partial_\ast(\Omega \Sigma)}\left( Y_w(X_1, \ldots, X_k) \right), 
\end{align*}
where the last wedge runs over all basic products in $k$ letters, and
the next to last step uses the change of variables
$m_i = \frac{n_i}{d}$: this gives a bijection between all
$(k+1)$-tuples $(n_1, \ldots, n_k, d)$ of positive integers with
$d \mid \gcd(n_1, \ldots, n_k)$, and all $(k+1)$-tuples
$(m_1, \ldots, m_k, d)$ of positive integers.

This in turn tells us, for $F_{\idop}$, that:
\[ F_{\idop}\left( \Sigma(X_1 \vee \cdots \vee X_k) \right) =
  \bigvee_{w \in W} F_{\idop}\left(\Sigma Y_w(X_1, \ldots, X_k) \right).\]
Plugging in $X_i = S^{d_i-1}$, for some $d_i \ge 2$, we get that
\[ F_{\idop}\left( S^{d_1} \vee \cdots \vee S^{d_k} \right) =
  \bigvee_{w \in W} F_{\idop}\left( S^{1+\sum_i m_i(w)(d_i-1)} \right), \]
  so that Proposition \ref{p:frQsLie} allows us to conclude Theorem
  \ref{t:homfree} for the wedge $S^{d_1} \vee \cdots \vee S^{d_k}$
  from the case of single spheres.

\subsection{The free spectral Lie algebra on a simply-connected space}

Bootstrapping from the previous cases to $F_{\idop}(X)$ for general
simply-connected $X$ is purely formal using the fact that
$H_\ast(F_{\idop}(X))$ only depends on the homology of $X$, as shown
in Proposition \ref{p:field}.

Let
$\phi_X : \frQsLie(\tilde{H}(X)) \to H_\ast(F_{\idop}(\Sigma^\infty
X))$
be the canonical map coming from the universal property of the the
free allowable $\Qs$-$\sLie$-algebra.

If $X$ is an arbitrary wedge of spheres, each of dimension at least
$2$, then we can write $X$ as a filtered colimit of finite wedges of
spheres and these fall under the previous case. Since homology and the free functors we are using all commute with filtered colimits, the result also holds for such an $X$.

Now, for a general simply-connected $X$, pick an $\F_2$-basis
$\{x_j\}$ of $\tilde{H}_\ast(X)$ and use it to construct an
equivalence of $\HF$-module spectra
$f : \bigvee_j \HF[|x_j|] \to \HF \sm \Sigma^\infty X$. The natural
transformation $\phi$ is a special case of a natural transformation
$\psi_V : \frQsLie(\pi_\ast(V)) \to \pi_\ast(F_{\HF \sm \idop}(V))$
for $\HF$-module spectra $V$, in the sense that
$\phi_X = \psi_{\HF \sm \Sigma^\infty X}$. In the naturality square
\[ \xymatrix{ \frQsLie(\pi_\ast(\bigvee_j \HF[|x_j|]))
  \ar[rr]^-{\psi_{(\bigvee_j \HF[|x_j|])}}
  \ar[d]_{\frQsLie(\pi_\ast(f))} && \pi_\ast(F_{\HF \sm \idop}(\bigvee_j \HF[|x_j|]))  \ar[d]^{\pi_\ast(F_{\HF \sm \idop}(f))} \\
  \frQsLie(\pi_\ast(\HF \sm \Sigma^\infty X)) \ar[rr]^-{\psi_{(\HF
  \sm \Sigma^\infty X)}} && \pi_\ast(F_{\HF \sm \idop}(\HF \sm
  \Sigma^\infty X)),}\]
all maps are known to be isomorphisms (the vertical ones because $f$
is an equivalence, the top one because it is
$\phi_{(\bigvee_j S^{|x_j|})}$) except the bottom one, which therefore
also is an isomorphism.

\section{Divided power algebras and Koszul duality}
\label{s:divpow}

In this section we review some results from \cite{Kuhn} and explain
heuristically their relation to our computation.

Let $R$ be an augmented $E_\infty$-ring spectrum, or more precisely,
an augmented commutative $S$-algebra. In \cite{Kuhn}, Kuhn discusses a
model for the Topological Andr\'e--Quillen homology spectrum $TAQ(R)$
of $R$ given by
\[TAQ(R) = \mathop{\mathrm{hocolim}}_{n \to \infty} \Omega^n S^n
  \otimes R\] where $\otimes$ denotes the tensoring of augmented
commutative $S$-algebras over pointed spaces (\cite[Definition
7.2]{Kuhn}). As mentioned in that paper, this model was shown by
Mandell to be equivalent to standard definitions of $TAQ(R)$.

Kuhn defines a filtration on $TAQ(R)$ (coming from a filtration on
these tensor products) and identifies \cite[Corollary 7.5]{Kuhn} the
associated graded pieces as
\[F_dTAQ(R)/F_{d+1}TAQ(R) \simeq \left(P_d \sm (R/S)^{\sm d}\right)_{h\Sigma_d},\]
where $P_d$ is the partition complex and $R/S$
denotes the cofiber of the unit $\eta : S \to R$ of the ring spectrum
$R$.

One can use the spectral sequence associated to this filtration to
compute the cohomology $H^\ast(TAQ(R); \F)$ and Kuhn identifies the
$E_1$-page \cite[Theorem 8.1]{Kuhn} for $\F$ of characteristic $p$ as
\[ E_1^{\ast,\ast}(TAQ(R);\F) = \cR(\Sigma L_r (\Sigma^ {-1}
  \tilde{H}^\ast (R; \F))),\]
where $L_r$ is the free restricted Lie algebra functor and $\cR$ is,
roughly speaking, the free algebra-over-the-Dyer-Lashof-algebra
functor.

\medskip

Before explaining how this result is related to this paper, a quick
reminder of divided power algebras is in order. Associated to an
operad $\cO$ one has the free algebra monad whose functor part is
\[X \mapsto \bigvee_{n \ge 0} \left(\cO(n) \sm X^{\sm
      n}\right)_{h\Sigma_n},\]
but also, under certain relatively mild conditions, another
monad defined using homotopy fixed points instead:
\[X \mapsto \bigvee_{n \ge 0} \left(\cO(n) \sm X^{\sm
      n}\right)^{h\Sigma_n}.\]
The algebras for this second monad are called \emph{divided power
  $\cO$-algebras} because of the special case of the commutative
operad in vector spaces over a field, where they are the traditional
divided power algebras. In the case of the classical Lie operad over
$\F_p$, the divided power algebras turn out to be restricted Lie
algebras (\cite[Theorem 0.1]{FresseDivPow}).

\medskip

Kuhn's result is related to the results in this paper through Koszul
duality: the spectral Lie operad should be thought of as Koszul dual
to the commutative operad in spectra, whose algebras are, of course,
$E_\infty$-ring spectra. Koszul duality of operads, roughly speaking,
should result in an equivalence between the category of algebras for
one operad and the category of divided power algebras for the other.
In this case $D \circ TAQ$ should implement the equivalence (where $D$
denotes Spanier--Whitehead duality), taking $E_\infty$-ring spectra to
divided power spectral Lie algebras ---which might be called
\emph{spectral restricted Lie algebras}, at least when working with
the version for $H\F_p$-module spectra.

The associated graded of Kuhn's filtration of $TAQ(R)$, namely,
\[\bigvee_{d \ge 0} \left(P_d \sm (R/S)^{\sm d}\right)_{h\Sigma_d}\]
looks a lot like the free $\idop$-algebra on $R/S$, except that it has
the partition complex $P_d$ instead of its Spanier--Whitehead dual
$\idop[n]$. The cohomology of the associated graded is given by the
homotopy groups of the mapping spectrum
\[\Map(\bigvee_{d \ge 0} \left(P_d \sm (R/S)^{\sm
      d}\right)_{h\Sigma_d}, H\F) \simeq
  \bigvee_{d \ge 0} \left(\idop[d] \sm \Map(R/S,H\F)\right)^{h
    \Sigma_d}\]
so one can think of Kuhn's formula for the $E_1$-page
as giving the homology of the free divided power $\idop$-algebra on
$R/S$. From this point of view it is no surprise that Kuhn's formula
involves free \emph{restricted} Lie algebras, while our formulas
involve free Lie algebras.

\end{document}